\documentclass[12pt]{amsart}

\usepackage{amssymb}

\usepackage{anysize}
\marginsize{2cm}{2cm}{2cm}{2cm}

\renewcommand{\ge}{\geqslant}

\usepackage{versions}
\usepackage{hyperref}
\hypersetup{
  colorlinks   = true, 
  urlcolor     = blue, 
  linkcolor    = blue, 
  citecolor   = red 
}
\usepackage{cancel}
\usepackage{yhmath}
\usepackage{enumitem}
\usepackage{stackengine}

\usepackage{lineno}

\newtheorem{theorem}{Theorem}[section]

\newtheorem{corollary}[theorem]{Corollary}

\theoremstyle{definition}

\theoremstyle{remark}

\newtheorem{example}[theorem]{Example}

\numberwithin{equation}{section}

\newcommand*{\where}{\ \ifnum\currentgrouptype=16 \middle\fi|\ }

\renewcommand{\epsilon}{\varepsilon}
\renewcommand{\phi}{\varphi}
\renewcommand{\kappa}{\varkappa}
\renewcommand{\theta}{\vartheta}

\def\R{{\mathbb R}}

\title[Topological lower bounds on the sizes of\dots]{Topological lower bounds on the sizes of simplicial complexes and simplicial sets}

\author{Sergey Avvakumov{$^\spadesuit$}}
\author{Roman Karasev{$^\clubsuit$}}

\thanks{{$^\spadesuit$} Supported by the European Research Council under the European Union's Seventh Framework Programme ERC Grant agreement ERC StG 716424 -- CASe}

\thanks{{$^\clubsuit$} Supported by the Russian Foundation for Basic Research grant 19-01-00169}

\address{Sergey Avvakumov, School of Mathematical Sciences, Tel Aviv University, Tel Aviv 69978, Israel}
\email{savvakumov@gmail.com}

\address{Roman Karasev, Institute for Information Transmission Problems RAS, Bolshoy Karetny per. 19, Moscow, Russia 127994 and Moscow Institute of Physics and Technology, Institutskiy per. 9, Dolgoprudny, Russia 141700
}
\email{r\_n\_karasev@mail.ru}
\urladdr{http://www.rkarasev.ru/en/}

\begin{document}

\maketitle

\begin{abstract}
We prove that if an $n$-dimensional space $X$ satisfies certain topological conditions then any triangulation of $X$ as well as any its representation as a simplicial set with contractible faces has at least $2^n$ faces of dimension $n$.

One example of such $X$ is the $n$-dimensional torus $(S^1)^n$.
\end{abstract}

\section{Introduction}



We establish a lower bound on the number of top-dimensional simplices in a simplicial complex or a simplicial set homeomorphic (or homotopically equivalent) to the given topological space, provided that the topology of the space is rich in a certain sense.

\begin{theorem}
\label{theorem:orbits}
Let $G$ be a finite group and let $V$ be an $n$-dimensional real $G$-representation. Let $X$ be an $n$-dimensional simplicial set with a simplicial action of $G$. Assume that

\begin{itemize}
\item[1.] the action is such that for every face $\sigma$ of $X$ and every $g\not=e\in G$ the intersection $\sigma\cap g\sigma$ is empty;
\item[2.] the image of every $G$-equivariant map $X\to V$ contains the origin in $V$ $($the Borsuk--Ulam type property$)$.
\end{itemize}

Then $X$ has at least $2^n$ $G$-orbits of faces of dimension $n$.
\end{theorem}



Here is an example of use of Theorem~\ref{theorem:orbits}:

\begin{corollary}
\label{corollary:torus-like}
Let $X$ be a simplicial set such that all its closed faces are contractible. Let $X$ have cohomology classes $\xi_1,\ldots,\xi_n\in H^1(X;\mathbb F_2)$ with non-zero product. Then $X$ has at least $2^n$ faces of dimension $n$.
\end{corollary}

\begin{example}
\label{example:torus}
In particular, any simplicial set with all faces contractible (this is a non-trivial requirement for a simplicial set) and homeomorphic (or homotopically equivalent) to the $n$-torus $(S^1)^n$ or to the $n$-dimensional real projective space $\R P^n$ has at least $2^n$ faces of dimension $n$.

For simplicial sets representing $\R P^n$ this bound is sharp. By identifying the opposite points of the boundary of the $(n+1)$-dimensional crosspolytope we get a simplicial set homeomorphic to $\R P^n$ with exactly $2^n$ faces of dimension $n$, all whose faces are embedded simplices and so contractible.
\end{example}

Theorem~\ref{theorem:orbits} can be seen as a generalization of the result of B\'ar\'any and Lov\'asz \cite{barany1982borsuk}, who proved that any centrally symmetric triangulation of the $n$-sphere has at least $2^{n+1}$ faces of dimension $n$. 


One powerful tool for face enumeration in manifolds (and simplicial complexes) is the manifold $g$-conjecture which was recently proved by Adiprasito \cite{adiprasito2018combinatorial}, see also \url{http://www.math.huji.ac.il/~adiprasito/bpa.pdf} for the most general statement of the manifold lower bound theorem, and see also \cite{papadakis2020characteristic} and \cite{adiprasito2021anisotropy} for another proof. Using the conjecture one can, under some assumptions, bound from below the number of faces in a triangulation of a manifold in terms of its Betti numbers, see \cite[Theorem 4.4]{govc2020many}. From personal communication with Karim Adiprasito we learned that the assumption of the coefficients field being infinite in \cite[Theorem 4.4]{govc2020many} is most probably not required. A finite field can be extended to an infinite one and Betti numbers are invariant under field extension. 

Compared to \cite[Theorem 4.4]{govc2020many}, our Theorem~\ref{theorem:orbits} can be applied to simplicial sets as well as simplicial complexes; it also has a rather short, elementary, and self-contained proof. On the other hand, \cite[Theorem 4.4]{govc2020many} sometimes gives stronger (though still only exponential) bounds, for instance for the $n$-torus, and involves Betti numbers which might be easier to deal with than checking the topological condition of Theorem~\ref{theorem:orbits}.

Another notable approach to general lower bounds on the size of a triangulation is to exploit the natural similarity between the number of top-dimensional simplices and the continuous \emph{volume}. The latter can be estimated in terms of the \emph{systole} using the classical Gromov's systolic inequality. For an example of this approach see \cite{kowalick2015combinatorial}. Unfortunately, systolic inequalities produce good bounds only when the systole is long enough. In our discrete setting the systole corresponds to the edge-length of the shortest non-contractible path in the given space and without additional assumptions can be as short as $3$ in simplicial complexes or $2$ in simplicial sets of the kind we study, which is not enough for a non-trivial bound.






\subsection*{Acknowledgments} We thank Karim Adiprasito for valuable discussions, Arseniy Akopyan for  discussions and drawing our attention to the argument below about random simplices, and the anonymous referees for numerous useful remarks.

\section{Proofs}

\begin{proof}[Proof of Theorem~\ref{theorem:orbits}]

Assuming that the number of $G$-orbits of $n$-faces of $X$ is strictly less than $2^n$, we are going to construct a $G$-equivariant map $X\to V$ that misses the origin contradicting Assumption 2 of the theorem. 

The proof basically follows the classical argument (which we learned from Arseniy Akopyan) showing that a probability that a random $n$-simplex contains the origin is $2^{-n}$, assuming that its vertices are mutually independent and the distribution of any single vertex is absolutely continuous and centrally symmetric.
	
Let $v_1,\ldots, v_N$ be the representatives of all the $G$-orbits of vertices of $X$. 
To construct a $G$-equivariant map $F : X\to V$ it is sufficient to choose the points $F(v_k)\in V$, extend this $F$ to other vertices of $X$ equivariantly (note that by Assumption 1 of the theorem the action of $G$ on vertices is free), and then extend $F$ to each face of $X$ linearly.
	
For any vector of signs $e\in \{-1,1\}^N$, we consider the modification of $F$, that is defined in the similar way, but starting with
\[
F_e (v_k) = e_k F(v_k). 
\]
Note that by equivariance $F_e(gv_k) = g e_k F(v_k) = e_k g F(v_k) = e_k F(g v_k)$.
	
Let us require that all such $F_e$ are \emph{generic}, that is the $F_e$ images of any $n$ vertices of $X$ from different $G$-orbits are linearly independent and the images of any $n+1$ vertices of $X$ from different $G$-orbits are affinely independent. The set of the initial maps $F$ that produce a non-generic $F_e$ for some sign vector $e$ is indeed a proper algebraic subset of the set of all possible maps $F$, hence we may restrict our consideration to families of generic $F_e$.
	
Let $\sigma$ be an $n$-face of $X$. Consider the case when it is degenerate in the sense that some of its vertices coincide in $X$ (this may happen with a simplicial set). Let $w_0,\ldots, w_k$ be its distinct vertices. By Assumption 1 they belong to different $G$-orbits and therefore by the genericity their $F_e$-images are linearly independent. Since $k<n$ in the presence of coincidences, $F_e(\sigma)$ does not touch the origin. 

Now consider the case when $\sigma$ has $n+1$ distinct vertices in $X$, all belonging to different $G$-orbits by Assumption 1. Let $w_0,\dots,w_n$ be the $F$-images of the vertices of $\sigma$, for a generic $F$ they are affinely independent. Write $0\in V$ as the unique (up to a multiplication by a non-zero number) linear combination 
\[
0=a_0w_0 + \dots + a_nw_n.
\]
For a generic $F$, we also have that $a_i\neq 0$ for all $i$. Evidently, $0\in F(\sigma)$ if and only if all the coefficients $a_i$ are of the same sign.

Choose $e\in \{-1,1\}^N$ uniformly at random. Considering the $F_e$-images $e_{k(0)}w_0,\dots,e_{k(n)}w_n$ of the vertices of $\sigma$ (here $k(i)$ is the number of the $G$-orbit of the $i$th vertex of $\sigma$), the linear combination above becomes

\begin{equation}
\label{equation:linear}
0 = a_0e_{k(0)} (e_{k(0)}w_0) + \dots + a_ne_{k(n)} (e_{k(n)}w_n),
\end{equation}
that is, its coefficients change to $e_{k(0)} a_0,\ldots,e_{k(n)} a_n$. Because all the vertices of $\sigma$ belong to distinct $G$-orbits, the indices $k(0), \ldots, k(n)$ are distinct and the corresponding signs $e_{k(i)}$ are all independent.

Tracing the signs of the coefficients $e_{k(0)} a_0,\ldots, e_{k(n)}a_n$, we conclude that precisely $2^{-n}$ fraction of the simplices $F_e(\sigma)$ have the coefficients in the linear combination \eqref{equation:linear} of the same sign, meaning that $F_e(\sigma)$ covers the origin with probability $2^{-n}$.

Choose the representatives $\sigma_1, \ldots, \sigma_m$ of $G$-orbits of $n$-faces of $X$. If $m<2^n$ then with positive probability none of $F_e(\sigma_i)$ contain the origin. From $G$-equivariance of $F_e$ there is no image of an $n$-face of $X$ containing the origin, that is $F_e(X)\not\ni 0$ at all, since the dimension of $X$ is $n$.  This contradiction shows that $m\ge 2^n$.
\end{proof}

\begin{proof}[Proof of Corollary~\ref{corollary:torus-like}]

In this proof we use some topology and the reader is referred to the textbook \cite{hatcher2001} that contains the necessary basics.

Every $\xi_i$ corresponds to a classifying map $f_i : X\to K(\mathbb Z_2,1)=\mathbb RP^\infty$ and the pullback of the double covering $S^\infty\to \mathbb RP^\infty$ is a double covering $\widetilde X_i\to X$. The fiber-wise product of these double coverings is a covering $\widetilde X\to X$ corresponding to a map $\pi_1(X)\to \mathbb Z_2^n$. Put $G=\mathbb Z_2^n$ and let $g_1,\ldots, g_n$ be its generators so that $g_i$ corresponds to the $i$th factor in the Cartesian product.

Note that a class $\xi_i$ is the pullback of the generator $\eta \in H^1(\mathbb RP^\infty;\mathbb F_2)$ under $f_i$. Consider $V=\mathbb R^n$ as a representation of $G=\mathbb Z_2^n$ so that $g_i$ flips the sign of the $i$th coordinate  of $\mathbb R^n$ and preserves the other coordinates. 

Now we want to apply Theorem~\ref{theorem:orbits} to $\widetilde X$ and the action of $G$ on it. Assumption 1 of Theorem~\ref{theorem:orbits} is satisfied since for every face $\sigma$ of $X$ (from its contractibility) its preimage in $\widetilde X$ is covered by a disjoint set of $G$-shifted copies of $\sigma$. Moreover, so we establish that such copies of all faces of $X$ in $\widetilde X$ constitute a representation of $\widetilde X$ as a simplicial set. 

As required by Theorem~\ref{theorem:orbits}, let us drop all faces of $X$ (and respectively $\widetilde X$) of dimension higher that $n$, that is, pass to the $n$-skeleton. This does not influence the cohomology product inequality $\xi_1\cdots \xi_n \neq 0$ and keeps Assumption 1 of Theorem~\ref{theorem:orbits} valid. It remains to show that there cannot be a $G$-invariant map $F: \widetilde X\to V$ for the $n$-dimensional $\widetilde X$, thus checking Assumption 2 of Theorem~\ref{theorem:orbits}. 

Put $\widetilde Z_i = \{x\in \widetilde X\ |\ F_i(x) = 0\}$ for $i=1,\ldots, n$, where $F_i$ are coordinates of $F$. Any set $\widetilde Z_i$ is $G$-invariant and projects to $Z_i\subseteq X$, let $U_i=X\setminus Z_i$. Note that the inequalities $F_i(x)>0$ and $F_i(x)<0$ split $\widetilde X\setminus \widetilde Z_i$ in two parts, interchanged by the involution $g_i\in G$. Projecting this splitting to the two-sheet covering $\widetilde X_i\to X$, we see that this covering $\widetilde X_i\to X$ trivializes over $U_i$. This means that the composition of the inclusion $U_i\to X$ and $f_i : X\to \mathbb RP^\infty$ is null-homotopic, which implies $\xi_i|_{U_i}=(f_i|_{U_i})^* \eta = 0$. 

If $U_1\cup\dots\cup U_n = X$ then one would have $\xi_1\dots \xi_n=0$ over $X$ by the standard property of the cohomology product. Since this is not the case, we have $Z_1\cap\dots\cap Z_n\neq \emptyset$. Because the sets $\widetilde Z_i$ are $G$-equivariant and so the lifting of each $Z_i$ in $\widetilde X$ is precisely $\widetilde Z_i$, we get that $\widetilde Z_1\cap\dots\cap \widetilde Z_n\neq \emptyset$, and so $F^{-1}(0)\neq\emptyset$.

\end{proof}

\bibliography{combi-crofton}
\bibliographystyle{abbrv}

\end{document}